\documentclass[11pt]{amsart}
\usepackage{amsmath}
\usepackage{amssymb}
\usepackage{dsfont}

\newtheorem{theorem}{Theorem}

\newtheorem{example}[theorem]{Example}

\email {arturbar@p.lodz.pl}
\email {malfil@math.uni.lodz.pl}
\email {szymonikemilka@wp.pl}

\input{tcilatex}

\begin{document}
\author{Artur Bartoszewicz}
\address{Institute of Mathematics, Technical University of \L \'od\'z,
W\'olcza\'nska 215, 93-005 \L \'od\'z, Poland}
\author{Ma\l gorzata Filipczak}
\address{Faculty of Mathematics and Computer Sciences, \L \'od\'z
University, Stefana Banacha 22, 90-238 \L \'od\'z, Poland}
\author{Emilia Szymonik}
\address{Institute of Mathematics, Technical University of \L \'od\'z,
W\'olcza\'nska 215, 93-005 \L \'od\'z, Poland}
\title[]{Multigeometric sequences and Cantorvals.}
\date{}

\begin{abstract}
For a sequence $x \in l_1 \setminus c_{00}$, one can consider the
achievement set $E(x)$ of all subsums of series $\sum_{n=1}^{\infty} x(n)$.
It is known that $E(x)$ is one of the following types of sets:

\begin{itemize}
\item finite union of closed intervals,

\item homeomorphic to the Cantor set,

\item homeomorphic to the set $T$ of subsums of $\sum_{n=1}^{\infty} c(n)$
where $c(2n-1)=\frac{3}{4^n}$ and $c(2n)=\frac{2}{4^n}$ (Cantorval).
\end{itemize}

Based on ideas of Jones and Velleman \cite{J} and Guthrie and Nymann \cite%
{GN} we describe families of sequences which contain, according to our
knowledge, all known examples of $x$'s with $E(x)$ being Cantorvals.
\end{abstract}

\maketitle

\section{Introduction}

Suppose that $x=\left( x\left( 0\right) ,x\left( 1\right) ,x\left( 2\right)
,\dots \right) $ is an absolutely summable sequence with infinitely many
nonzero terms (i.e. $x\in l_{1}\setminus c_{00}$) and let%
\begin{equation*}
E\left( x\right) =\left\{ \underset{n=1}{\overset{\infty }{\sum }}%
\varepsilon _{n}x\left( n\right) :\varepsilon _{n}\in \left\{ 0,1\right\}
\right\}
\end{equation*}%
denote the set of all subsums of the series $\underset{n=1}{\overset{\infty }%
{\sum }}x\left( n\right) ,$ called the achievement set of $x$. It is easily
seen that for $x=\left( \frac{2}{3},\frac{2}{3^{2}},\frac{2}{3^{3}},\dots
\right) $ the set $E\left( x\right) $ is equal to the Cantor tenary set C,
and for $x=\left( \frac{1}{2},\frac{1}{2^{2}},\frac{1}{2^{3}},...\right) $
we have $E\left( x\right) =[0,1]$.

Achievement sets have been consider by many authors, some results have been
proved several times (see, for example, \cite{K} and \cite{H}) and even
conjectures formulated, despite the fact that suitable counterexamples had
been earlier published (compare \cite{GN}, \cite{KL} and \cite{P}).
Recently, an interesting survey of properties of achievement sets for
various (even divergent) sequences was presented by Rafe Jones in \cite{J}.
This paper is motivated, in particular, by the example from \cite{J} (due to
Velleman and Jones), which will be described more precisely in Theorem \ref%
{th2} and Example \ref{ex-JV}.

The following properties of sets $E(x)$ were described in 1914 by S. Kakeya
in \cite{K}:

\begin{itemize}
\item[I.] $E(x)$ is a compact perfect set.

\item[II.] If $|x(n)| > \sum_{i>n} |x(i)|$ for $n$ sufficiently large, then $%
E(x)$ is homeomorphic to the Cantor set $C$.

\item[III.] If $|x(n)| \leqslant \sum_{i>n} |x(i)|$ for $n$ sufficiently
large, then $E(x)$ is a finite union of closed intervals. Moreover, if $%
|x(n)| \geqslant |x(n+1)|$ for almost all $n$ and $E(x)$ is a finite union
of closed intervals, then $|x(n)| \leqslant \sum_{i>n} |x(i)|$ for $n$
sufficiently large.
\end{itemize}

In the same paper Kakeya formulated the hypothesis that, for any $x\in
l_{1}\setminus c_{00}$, the set $E(x)$ is homeomorphic to $C$ or is a finite
union of closed intervals. In 1980 it was shown that the Kakeya conjecture
is false \cite{WS}. We recall a number of examples in the literature which
demonstrate the falseness of the conjecture. We use the original notations
proposed by the authors. The notation will be unified later in the paper.

A. D. Weinstein and B. E. Shapiro in \cite{WS} gave an example of a sequence 
$a$ with $a(n)\geqslant a(n+1)>0$ for all $n$, and $a(n)>\sum_{i>n}a(i)$ for
infinitely many $n$ (hence $E(a)$ is not a finite union of intervals), but
having the property that the set $E(a)$ contains an interval. The sequence $%
a $ is defined by the formulas: $a(5n+1)=0,24\cdot 10^{-n}$, $%
a(5n+2)=0,21\cdot 10^{-n}$, $a(5n+3)=0,18\cdot 10^{-n}$, $a(5n+4)=0,15\cdot
10^{-n}$, $a(5n+5)=0,12\cdot 10^{-n}.$ So, 
\begin{equation*}
a=(\frac{3\cdot 8}{10},\frac{3\cdot 7}{10},\frac{3\cdot 6}{10},\frac{3\cdot 5%
}{10},\frac{3\cdot 4}{10},\frac{3\cdot 8}{100},\dots ).
\end{equation*}%
However, they did not justify why the interior of $E(a)$ is non-empty.

Independently, C. Ferens (\cite{F}) constructed a sequence $b$ such that $%
E\left( b\right) $ is not a finite union of intervals but contains an
interval, putting $b(5l-m)=(m+3)\frac{2^{l-1}}{3^{3l}}$ for $m=0,1,2,3,4$
and $l=1,2,\dots $. Therefore 
\begin{equation*}
b=(7\cdot \frac{1}{27},6\cdot \frac{1}{27},5\cdot \frac{1}{27},4\cdot \frac{1%
}{27},3\cdot \frac{1}{27},7\cdot \frac{2}{27^{2}},\dots ).
\end{equation*}

J. A. Guthrie and J. E. Nymann gave a simpler example of a sequence which
achievement set is not a finite union of closed intervals and is not
homeomorphic to the Cantor set, defining a sequence by formulas:

$c(2n-1)= \frac{3}{4^n}$ and $c(2n)= \frac{2}{4^n}$ for $n=1,2, \dots.$

In a series of papers \cite{GN}, \cite{NS1} and \cite{NS2} J. E. Nymann with
J. A. Guthrie and R. A. S\'{a}enz characterized the topological structure of
the set of subsums of infinite series in the following manner:

\begin{theorem}
\label{th1} For any $x \in l_1 \setminus c_{00}$, the set $E(x)$ is one of
the following types:

\begin{itemize}
\item[(i)] a finite union of closed intervals;

\item[(ii)] homeomorphic to the Cantor set;

\item[(iii)] homeomorphic to the set $E(c)$ (of subsums of the sequence

$(\frac{3}{4}, \frac{2}{4}, \frac{3}{16}, \frac{2}{16}, \frac{3}{64}, \dots) 
$).
\end{itemize}
\end{theorem}

Note, that the set $E(c)$ is homeomorphic to $C\cup \bigcup_{n=1}^{\infty
}S_{2n-1}$, where $S_{n}$ denotes the union of the $2^{n-1}$ open middle
thirds which are removed from $[0,1]$ at the $n$-th step in the construction
of the Cantor ternary set $C$. Such sets are called Cantorvals (to emphasize
their similarity to unions of intervals and to the Cantor set
simultaneously). Formally, a Cantorval (more precisely, an $\mathcal{M}$%
-Cantorval - compare \cite{MO}) is a non-empty compact subset $S$ of the
real line such that $S$ is the closure of its interior, and both endpoints
of any component with non-empty interior are accumulation points of
one-point components of $S$.

Theorem \ref{th1} states that the space $l_{1}$ can be decomposed into four
sets $c_{00}$, $\mathcal{C}$, $\mathcal{I}$ and $\mathcal{MC}$, where $%
\mathcal{I}$ consists of sequences $x$ with $E(x)$ equal to a finite union
of intervals, $\mathcal{C}$ consists of sequences $x$ with $E(x)$
homeomorphic to the Cantor set, and $\mathcal{MC}$ consists of sequences $x$
with $E(x)$ being Cantorvals. Some algebraic properties and topological
(Borel) classification of these subsets of $l_{1}$ have been recently
discussed in \cite{BBGS}.

Finally, in Jones' paper \cite{J} there is presented a sequence 
\begin{equation*}
d=(\frac{3}{5},\frac{2}{5},\frac{2}{5},\frac{2}{5},\frac{3}{5}\cdot \frac{19%
}{109},\frac{2}{5}\cdot \frac{19}{109},\frac{2}{5}\cdot \frac{19}{109},\frac{%
2}{5}\cdot \frac{19}{109},\frac{3}{5}\cdot (\frac{19}{109})^{2},\dots ).
\end{equation*}%
In \cite{J}, R. Jones shows a continuum of sequences generating Cantorvals,
indexed by a parameter $q$, by proving that, for any positive number $q$
with 
\begin{equation*}
\frac{1}{5}\leqslant \sum_{n=1}^{\infty }q^{n}<\frac{2}{9}
\end{equation*}%
(i.e. $\frac{1}{6}\leqslant q<\frac{2}{11}$) the sequence 
\begin{equation*}
(\frac{3}{5},\frac{2}{5},\frac{2}{5},\frac{2}{5},\frac{3}{5}q,\frac{2}{5}q,%
\frac{2}{5}q,\frac{2}{5}q,\frac{3}{5}q^{2},\dots )
\end{equation*}%
is not in $\mathcal{C}$ nor $\mathcal{I}$, so it belongs to $\mathcal{MC}$.
Based on Jones' idea, we will describe one-parameter families of sequences
which contain (in particular) $a,b,d$ and many others.

\section{The main result.}

For any $q\in (0,\frac{1}{2})$ we will use the symbol $(k_{1},k_{2},\dots
,k_{m};q)$ to denote the sequence $(k_{1},k_{2},\dots
,k_{m},k_{1}q,k_{2}q,\dots ,k_{m}q,k_{1}q^{2},k_{2}q^{2},\dots
,k_{m}q^{2},\dots )$. Such sequences we will call multigeometric.

\begin{theorem}
\label{th2} Let $k_{1}\geqslant k_{2}\geqslant \dots \geqslant k_{m}$ be
positive integers and $K=\sum_{i=1}^{m}k_{i}$. Assume that there exist
positive integers $n_{0}$ and $n$ such that each of numbers $%
n_{0},n_{0}+1,\dots ,n_{0}+n$ can be obtained by summing up the numbers $%
k_{1},k_{2},\dots ,k_{m}$ (i.e. $n_{0}+j=\sum_{i=1}^{m}\varepsilon _{i}k_{i}$
with $\varepsilon _{i}\in \{0,1\},j=1,\dots ,n$). If\textbf{\ }$q\geqslant 
\frac{1}{n+1}$\textbf{\ }then\textbf{\ }$E(k_{1},\dots ,k_{m};q)$\textbf{\ }%
has a nonempty interior\textbf{. }If\textbf{\ }$q<\frac{k_{m}}{K+k_{m}}$%
\textbf{\ }then\textbf{\ }$E(k_{1},\dots ,k_{m};q)$\textbf{\ }is not a
finite union of intervals. Consequently, if 
\begin{equation*}
\frac{1}{n+1}\leqslant q<\frac{k_{m}}{K+k_{m}}
\end{equation*}%
then $E(k_{1},\dots ,k_{m};q)$ is a Cantorval.
\end{theorem}

\begin{proof}
Denote $x_{q}=(k_{1},\dots ,k_{m};q)$. We start with showing that, for $q<%
\frac{k_{m}}{K+k_{m}}$, $E(x_{q})$ is not a finite union of closed
intervals. Observe first, that the sequence $x_{q}$ is non-increasing.
Indeed, from the inequality $q<\frac{k_{m}}{K+k_{m}}$, it follows that $%
qK+qk_{m}<k_{m}$, and 
\begin{equation*}
k_{m}>\frac{qK}{1-q}>qK>qk_{1}.
\end{equation*}%
Moreover, using the same inequality, we obtain 
\begin{equation*}
\sum_{i>m}x_{q}(i)=K\sum_{j=1}^{\infty }q^{j}=K\frac{q}{1-q}<k_{m}.
\end{equation*}%
Hence, for any $n\in \mathbb{N}$, we have $x_{q}(nm)>\sum_{i>nm}x_{q}(i)$
and, according to the second sentence of the Kakeya property III, $E(x_{q})$
is not a finite union of closed intervals.

Suppose now that $q\geqslant \frac{1}{n+1}$\ and consider the sequence 
\begin{equation*}
y=(1,\dots ,1,q,\dots ,q,q^{2},\dots ,q^{2},\dots )
\end{equation*}%
with $n$ repetitions of each term. Note that, for any $k\in \mathbb{N}$, the
sum 
\begin{equation*}
\sum_{j>nk}y(j)=q^{k-1}\frac{nq}{1-q}
\end{equation*}%
is, by inequality $q\geqslant \frac{1}{n+1}$, bigger than or equal to $%
y(nk)=q^{k-1}.$ Therefore, for any $i\in \mathbb{N}$ 
\begin{equation*}
y(i)\leqslant \sum_{j>i}y(j)
\end{equation*}%
and again from the property III, we obtain that $E(y)$ has non-empty
interior. To end the proof, we show that 
\begin{equation*}
n_{0}\sum_{j=0}^{\infty }q^{j}+E(y)\subset E(x_{q}).
\end{equation*}%
If $t\in n_{0}\sum_{j=0}^{\infty }q^{j}+E(y)$, then there exist $p_{i}\in
\{0,1,\dots ,n\}$, $i=0,1,2,\dots $ such that 
\begin{equation*}
t=(n_{0}+n_{0}q+\dots )+(p_{0}+p_{1}q+\dots ).
\end{equation*}%
Therefore 
\begin{equation*}
t=(n_{0}+p_{0})+(n_{0}+p_{1})q+\dots
\end{equation*}%
belongs to $E(x_{q})$.
\end{proof}

\section{Examples.}

Using the latter theorem, we can easily check that sequences $a,b$ and $d$
generate Cantorvals, because they belong to appropriate one-parameter
families, indexed by $q$.

\begin{example}
\label{ex-WS}The Weinstein-Shapiro sequence (\cite{WS}).

It is clear that if $E(x)$ is a Cantorval, $\alpha \neq 0$ and $\alpha x =
(\alpha x(1), \alpha x(2), \dots)$, then $E(\alpha x)$ is a Cantorval too.
To simplify a notation we multiply the sequence $a$ by $\frac{10}{3}$ and
consider the family of sequences 
\begin{equation*}
a_q=(8,7,6,5,4;q)
\end{equation*}
for $q \in (0, \frac{1}{2})$. Summing up $8,7,6,5$ and $4$, we can get any
natural number between $n_0=4$ and $n+n_0=26$. Therefore, by Theorem \ref%
{th2}, for any $q$ satisfying inequalities 
\begin{equation*}
\frac{1}{23} \leqslant q < \frac{4}{34},
\end{equation*}
the sequence $a_q$ generates a Cantorval. Obviously, the number $\frac{1}{10}
$ used in \cite{WS} belongs to $[\frac{1}{23}, \frac{4}{34}).$ It is not
difficult to check (using III) that $a_q \in \mathcal{I}$ for $q \geqslant 
\frac{4}{34}.$
\end{example}

\begin{example}
\label{ex-F}The Ferens sequence (\cite{F}).

For the family of sequences 
\begin{equation*}
b_{q}=(7,6,5,4,3;q)
\end{equation*}%
$K$ is equal to $25$, $n_{0}=3$ and $n=19$. Hence, for any $q\in \lbrack 
\frac{1}{20},\frac{3}{28})$, $b_{q}$ generates a Cantorval. In particular,
the sequence $(7,6,5,4,3;\frac{2}{27})$, obtained from the Ferens sequence
by multiplication by a constant, generates a Cantorval. Note that $b_{q}\in 
\mathcal{I}$, for $q\geqslant \frac{3}{28}$.
\end{example}

\begin{example}
\label{ex-JV} The Jones-Velleman sequence (\cite{J}).

Applying Theorem \ref{th2} to the sequence 
\begin{equation*}
d_{q}=(3,2,2,2;q)
\end{equation*}%
we obtain $K=9$, $n_{0}=2$ and $n=5$, so for any $q\in \lbrack \frac{1}{6},%
\frac{2}{11})$, $E(d_{q})$ is a Cantorval. Moreover $d_{q}\in \mathcal{I}$
for $q\geqslant \frac{2}{11}.$

We can also consider analogous sequences for more than three $2$'s. In fact,
any sequence 
\begin{equation*}
x_{q}=(3,\underbrace{2,\dots ,2}_{k-times};q)
\end{equation*}%
with $q\in \lbrack \frac{1}{2k},\frac{2}{2k+5})$, generates a Cantorval.

Note that for $k=1$ and $k=2$ the argument of Theorem \ref{th2} breaks down,
because $\frac{1}{2k}>\frac{2}{2k+5}$. It means, in particular, that Theorem %
\ref{th2} does not apply to the Guthrie and Nymann example $c=\left( 3,2;%
\frac{1}{4}\right) $.
\end{example}

However, we can apply Theorem \ref{th2} to "shortly defined" sequences.
Indeed, for the sequence $(4,3,2;q)$, numbers $K$, $n_0$ and $n$ are the
same as for $d_q$.

It is not difficult to check that, to keep the interval $[\frac{1}{n+1},%
\frac{k_m}{K+k_m})$ non-empty, $m$ should be greater than $2$.

There is a natural question if Theorem \ref{th2} precisely describes the set
of $q$ with $(k_1, \dots, k_m;q) \in \mathcal{MC}$. The upper bounds, for
all mentioned examples are exact, because $(k_1, \dots, k_m;q) \in \mathcal{I%
}$, for $q> \frac{k_m}{K+k_m}$. However, this is not true for all sequences
satisfying the assumptions of Theorem \ref{th2}.

\begin{example}
\label{ex-h}For the sequence $h_{q}=(10,9,8,7,6,5,2;q)$, we have $K=47$, $%
n_{0}=5$ and $n=37$. Therefore the interval $[\frac{1}{n+1},\frac{k_{m}}{%
K+k_{m}})=[\frac{1}{38},\frac{2}{49})$ is non-empty.

However, for $h=(10,9,8,7,6,5,2;\frac{2}{49})$ and any $n \in \mathbb{N}$,
we have $\sum_{i>7n-1} h(i)= (\frac{2}{49})^{n-1}(2+ \frac{\frac{2}{49}
\cdot 47}{1- \frac{2}{49}})=4(\frac{2}{49})^{n-1}< h(7n-1)$. It means that $%
h \notin \mathcal{I}$. Since $\frac{2}{49} > \frac{1}{38}$, we have $h
\notin \mathcal{C}$ and so $h \in \mathcal{MC}$.

It is not difficult to check, using III again, that $h_{q}\notin \mathcal{I}$
if and only if $q<\frac{3}{50}$.
\end{example}

Observe, that $E(k_{1},\dots ,k_{m};q)\subset \sum_{i=1}^{K}C_{q}$, where $%
C_{q}=E((1;q))$ and $\sum_{i=1}^{K}C_{q}$ denotes the algebraic sum. In \cite%
{C} it is proved that, if $q<\frac{1}{K+1}$ then $\sum_{i=1}^{K}C_{q}$ is
homeomorphic to the Cantor set. The following theorem improves this result.

\begin{theorem}
\label{th3}Let $x=\left( k_{1},...,k_{m};q\right) $ be a multigeometric
sequence and 
\begin{equation*}
\Sigma :=\left\{ \underset{i=1}{\overset{m}{\sum }}\varepsilon
_{i}k_{i}:\left( \varepsilon _{i}\right) _{i=1}^{m}\in \left\{ 0,1\right\}
^{m}\right\} \text{.}
\end{equation*}%
If $q<1/card\left( \Sigma \right) $ then $E\left( x\right) $ is a Cantor set.

\begin{proof}
Clearly, $E\left( x\right) =\Sigma +qE\left( x\right) $. Suppose that $%
q<1/card\left( \Sigma \right) $ and the set $E\left( x\right) $ has a
nonempty interior. Therefore $E\left( x\right) $ has positive Lebesgue
measure $\lambda \left( E\left( x\right) \right) $ and%
\begin{equation*}
\lambda \left( E\left( x\right) \right) \leq card\left( \Sigma \right) \cdot
q\cdot \lambda \left( E\left( x\right) \right) <\lambda \left( E\left(
x\right) \right)
\end{equation*}%
which gives a contradiction.
\end{proof}
\end{theorem}

Using the latter theorem to the Weinstein-Shapiro sequence $a_{q}=\left(
8,7,6,5,4;q\right) $ (compare Example \ref{ex-WS}) we obtain $\Sigma $ of
cardinality $25$. It means that $E\left( a_{q}\right) \in \mathcal{C}$ for $%
q\in \left( 0,\frac{1}{25}\right) $.\ We do not know what is the type of $%
E(a_{q})$ for $q\in \left[ \frac{1}{25},\frac{1}{23}\right) $. Analogously, $%
E\left( b_{q}\right) \in \mathcal{C}$ for $q\in \left( 0,\frac{1}{22}\right) 
$ (compare Example \ref{ex-F}), $E\left( d_{q}\right) \in \mathcal{C}$ for $%
q\in \left( 0,\frac{1}{8}\right) $ (compare Example \ref{ex-JV}) and $%
E\left( h_{q}\right) \in \mathcal{C}$ for $q\in \left( 0,\frac{1}{42}\right) 
$ (compare Example \ref{ex-h}).

\section{The generalization of Guthrie-Nymann example.}

We have just mentioned that Theorem \ref{th2} does not work for sequences $%
(3,2;q)$ and $(3,2,2;q)$. However, Guthrie and Nymann have proved that $%
c=(3,2;\frac{1}{4})\in \mathcal{MC}$. Following their method we will find $q<%
\frac{1}{n+1}$ such that 
\begin{equation*}
(3,\underbrace{2,\dots ,2}_{k-times};q)\in \mathcal{MC}.
\end{equation*}

\begin{theorem}
\label{th7}For any sequence of the form 
\begin{equation*}
x_{k}=(3,\underbrace{2,\dots ,2}_{k-times};\frac{1}{2k+2}),
\end{equation*}%
the set $E(x_{k})$ is a Cantorval.
\end{theorem}

\begin{proof}
We know that $x_k \notin \mathcal{I}$, because $\frac{1}{2k+2}< \frac{2}{2k+5%
}$ (compare with Example \ref{ex-JV}). It remains to prove that $E(x_k)$
contains an interval.

For a sake of clarity, we will prove a thesis for $k=2,$ i.e. we will show
that $E(x_{2})\supset \lbrack 3,4]$, which means that any point 
\begin{equation*}
t=3+\sum_{i=1}^{\infty }\frac{\varepsilon _{i}}{6^{i}}
\end{equation*}%
with $\varepsilon _{i}=\{0,\dots ,5\}$ belongs to $E(x_{2})$.

Since $E(x_2)$ is closed and the set $\{3+ \sum_{i=1}^{n} \frac{\varepsilon_i%
}{6^i} : \varepsilon_i \in \{0, \dots,5 \}, i \leqslant n, n=0,1, \dots \}$
is dense in $[3,4]$, it is enough to show that 
\begin{equation*}
3+ \sum_{i=1}^{n} \frac{\varepsilon_i}{6^i} \in E(x_2)
\end{equation*}
for any $n=0,1, \dots$, $\varepsilon_i =0, \dots ,5$.

For $n=0$, we have $3 \in E(x_2)$.

Suppose that any number of the form

\begin{equation}  \label{eq1}
t^{\prime }= 3+ \sum_{i=1}^{n-1} \frac{\varepsilon_i}{6^i}
\end{equation}
belongs to $E(x_2)$. It means that there exist $a_i, b_i, c_i \in \{0,1 \}$
such that 
\begin{equation*}
t^{\prime }= 3+ \sum_{i=1}^{n-1} \frac{3a_i + 2b_i +2 c_i}{6^i}.
\end{equation*}
Let 
\begin{equation*}
t = 3+ \sum_{i=1}^{n} \frac{\varepsilon_i}{6^i}.
\end{equation*}
If $\varepsilon_n =0,2,3,4$ or $5$, then 
\begin{equation*}
t=t^{\prime }+\frac{3a_n + 2b_n +2 c_n}{6^n}
\end{equation*}
for some $t^{\prime }$ and suitable $a_n, b_n$ and $c_n$.

Suppose that $\varepsilon _{n}=1$. Hence 
\begin{equation*}
t=t^{\prime }+\frac{1}{6^{n}}=t^{\prime }-\frac{1}{6^{n-1}}+\frac{3+2+2}{%
6^{n}}=t^{\prime \prime }+\frac{3+2+2}{6^{n}}.
\end{equation*}%
If $t^{\prime }>3$ then $t^{\prime \prime }$ satisfies (\ref{eq1}) and the
proof is complete. If $t^{\prime }=3$ then 
\begin{equation*}
t=3+\frac{1}{6^{n}}=2+(1-\frac{1}{6^{n-1}})+\frac{7}{6^{n}}=2+(\frac{5}{6}+%
\frac{5}{6^{2}}+\dots +\frac{5}{6^{n-1}})+\frac{3+2+2}{6^{n}}\in E(x_{2}).
\end{equation*}%
To show that for a fixed $k\geqslant 2$, any $n=0,1,...$ and $\varepsilon
_{i}=0,...,2k+1$%
\begin{equation*}
3+\sum_{i=1}^{n}\frac{\varepsilon _{i}}{\left( 2k+2\right) ^{i}}\in E(x_{k})
\end{equation*}%
and hence $\left[ 3,4\right] \in E(x_{k})$, one can repeat the previous
considerations, using the equality%
\begin{equation*}
3+\frac{1}{\left( 2k+2\right) ^{n}}=2+\left( 1-\frac{1}{\left( 2k+2\right)
^{n-1}}\right) +\frac{3+2k}{\left( 2k+2\right) ^{n}}\text{.}
\end{equation*}
\end{proof}

Note that, even for special sequences considered in this paper, it is very
hard to distinguish sequences belonging to $\mathcal{C}$ from sequences
belonging to $\mathcal{MC}$. In particular, for any sequence of the form 
\begin{equation*}
x_{q}=(3,2,\dots ,2;q),
\end{equation*}%
where $2$'s repeats itself $k$-times, $x_{q}\in \mathcal{I}$ if and only if $%
q\geqslant \frac{2}{2k+5}$, and, by Theorem \ref{th3}, $x_{q}\in \mathcal{C}$
for $q<\frac{1}{card\left( \Sigma \right) }=\frac{1}{2k+2}$.

\setlength{\unitlength}{1mm} 
\begin{picture}(96,38)
\put(0,12){\vector(1,0){123}}
\put(2,14){0}
\put(13,6){$\mathcal{C}$}
\put(25,16){$\frac{1}{2k+2}$}
\put(25,6){$\mathcal{MC}$}
\put(28,11.75){\line(1,0){25}}
\put(41,7.5){?}
\put(51,16){$\frac{1}{2k}$}
\put(61,6){$\mathcal{MC}$}
\put(75,16){$\frac{2}{2k+5}$}
\put(89,6){$\mathcal{I}$}
\put(102,16){$\frac{1}{2}$}
\multiput(3,12)(25,0){5}{\circle*{1}}
\end{picture}

We have no idea what are the types of sets $E(x_{q})$ for $q\in (\frac{1}{%
2k+2},\frac{1}{2k})$.

Finally, go back to the Guthrie and Nymann sequence $c=\left( 3,2;\frac{1}{4}%
\right) $. Z. Nitecki, in \cite{N}, proved that for $q<\frac{1}{4}$ the
sequence%
\begin{equation*}
c_{q}=\left( 3,2;q\right)
\end{equation*}%
belongs to $\mathcal{C}$. The same conclusion follows easily from Theorem %
\ref{th3}. It is not difficult to check that $x_{q}\in \mathcal{I}$ if and
only if $q\geqslant \frac{2}{7}$.

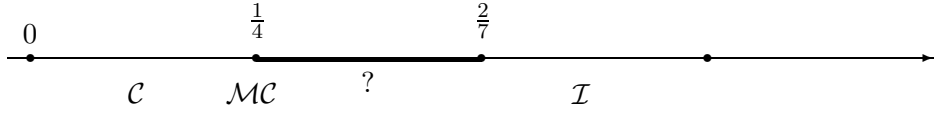
\setlength{\unitlength}{1mm} 
\begin{picture}(96,38)
\put(0,12){\vector(1,0){123}}
\put(2,14){0}
\put(16,6){$\mathcal{C}$}
\put(32,16){$\frac{1}{4}$}
\put(33,11.7){\line(1,0){30}}
\put(47,7.5){?}
\put(29,6){$\mathcal{MC}$}
\put(62,16){$\frac{2}{7}$}
\put(75,6){$\mathcal{I}$}
\multiput(3,12)(30,0){4}{\circle*{1}}
\end{picture}

We do not know what is the type of $E(x_q)$ for $q \in (\frac{1}{4}, \frac{2%
}{7})$.

At last, let us consider one more example from \cite{N} (due to Kenyon).

\begin{example}
The achievement set $E(f)$ of the sequence $f=(6,1;\frac{1}{4})$ (in our
notation) is $\mathcal{M}$-Cantorval. To prove it, Nitecki observes that $6$
is equal to $2$ mod $4$ and each element of $\mathbb{Z}_{4}$ can be obtained
by summing up the numbers 2 and 1 (compare the proof of Theorem \ref{th7}).
Then he makes use of the Baire category theorem. By our mind, this fact can
be explained in a much simpler way. Indeed, 
\begin{equation*}
f=\frac{1}{2}(12,2;\frac{1}{4})=\frac{1}{2}(12,2,3,2\cdot \frac{1}{4},3\cdot 
\frac{1}{4},2\cdot \frac{1}{16},\dots ).
\end{equation*}%
Hence 
\begin{equation*}
E(f)=\frac{1}{2}E(12,3,2,3\cdot \frac{1}{4},2\cdot \frac{1}{4},\dots )=\frac{%
1}{2}E(c)\cup \frac{1}{2}(E(c)+12)
\end{equation*}%
and $E(f)$ is of the same form as $E(c)$. In general, it is easy to observe
(in the same way as above) that the sequences $(k_{1},k_{2},\dots ,k_{m};q)$
and $(q^{n_{1}}k_{1},q^{n_{2}}k_{2},\dots ,$ $q^{n_{m}}k_{m};q)$ for
integers $n_{1},n_{2},\dots ,n_{m}$ are in the same set among of $\mathcal{C}%
,\mathcal{I}$ or $\mathcal{MC}$. Observe, for instance, that $(2,1;\frac{1}{4%
})\in \mathcal{I}$ and $(3,8;\frac{1}{4})\in \mathcal{MC}$. However, each
element of $\mathbb{Z}_{4}$ can be obtained by summing up $2$ and $1$, but $%
2 $ can not be obtained by summing up $3$ and $8$.
\end{example}

\section*{Acknowledgement}

The authors are greatly indebted to the referees for several helpful
comments improving the paper. In particular they wish to express their
thanks for the idea of Theorem \ref{th3}.

\end{document}